\documentclass[11pt]{amsart}
\newcommand{\preprint}[1]{}

\newcommand{\hide}[1]{}

\usepackage{amssymb}
\usepackage{amsbsy}
\usepackage{amscd}
\usepackage{amsmath}
\usepackage{amsthm}

\input xy
\xyoption{all}

\numberwithin{equation}{section}

\theoremstyle{plain}
\newtheorem{thm}{Theorem}[section]
\newtheorem{prop}[thm]{Proposition}
\newtheorem{claim}[thm]{Claim}

\newtheorem{lem}[thm]{Lemma}

\theoremstyle{definition}
\newtheorem{defi}[thm]{Definition}

\theoremstyle{remark}
\newtheorem{example}[thm]{Example}

\newcommand{\A}{{\mathcal A}}

\newcommand{\D}{{\mathcal D}}
\newcommand{\E}{{\mathcal E}}

\newcommand{\I}{{\mathcal I}}

\newcommand{\LB}{{\mathcal L}}

\newcommand{\fM}{{\mathfrak M}}

\newcommand{\PP}{{\mathbb P}}

\renewcommand{\S}{{\mathcal S}}
\newcommand{\fS}{{\mathfrak S}}

\newcommand{\U}{{\mathcal U}}
\newcommand{\V}{{\mathcal V}}

\newcommand{\X}{{\mathcal X}}
\newcommand{\Y}{{\mathcal Y}}
\newcommand{\Z}{{\mathcal Z}}

\newcommand{\Integers}{{\mathbb Z}}
\newcommand{\ComplexNumbers}{{\mathbb C}}

\newcommand{\linsys}[1]{{\mid}#1{\mid}}

\newcommand{\RightArrowOf}[1]{\stackrel{#1}{\rightarrow}}

\newcommand{\StructureSheaf}[1]{{\mathcal O}_{#1}}

\newcommand{\restricted}[2]{#1_{\mid_{#2}}}

\newcommand{\rank}{{\rm rank}}

\newcommand{\Pic}{{\rm Pic}}
\newcommand{\Sym}{{\rm Sym}}
\newcommand{\Ext}{{\rm Ext}}

\newcommand{\Hom}{{\rm Hom}}

\newcommand{\End}{{\rm End}}

\newcommand{\SheafHom}{{\mathcal H}om}
\newcommand{\SheafEnd}{{\mathcal E}nd}

\newcommand{\SheafExt}{{\mathcal E}xt}

\newcommand{\Wedge}[1]{\stackrel{#1}{\wedge}}

\begin{document}
\title[Stability]
{Stability of a natural sheaf over the cartesian square of the Hilbert scheme of points on a $K3$ surface}
\author{Eyal Markman}
\address{Department of Mathematics and Statistics, 
University of Massachusetts, Amherst, MA 01003, USA}
\email{markman@math.umass.edu}

\date{\today}

\begin{abstract}
Let $S$ be a $K3$ surface and $S^{[n]}$ the Hilbert scheme of length $n$ subschemes of $S$. Over the cartesian square 
$S^{[n]}\times S^{[n]}$
there exists a natural reflexive rank $2n-2$ coherent sheaf $E$, which is locally free away from the diagonal. The fiber of $E$ over 
the point $(I_{Z_1},I_{Z_2})$, corresponding to ideal sheaves of distinct subschemes $Z_1\neq Z_2$, is
$\Ext^1_S(I_{Z_1},I_{Z_2})$. We prove that $E$ is slope stable if the rank of the Picard group of $S$ is $\leq 19$.
The Chern classes of $\SheafEnd(E)$ are known to be monodromy invariant.
Consequently, the sheaf $\SheafEnd(E)$ is polystable-hyperholomorphic.

\end{abstract}

\maketitle
\tableofcontents
%
\section{Introduction}

Let $S$ be a  $K3$ surface, not necessarily projective, and $n$  an integer $\geq 2$. 
Denote by $S^{[n]}$ the Hilbert scheme, or Douady space, of length $n$ zero-dimensional subschemes of $S$.
Let $\U$ be the ideal sheaf of the universal subscheme of $S\times S^{[n]}$. 
Let $\pi_{ij}$ be the projection from $S^{[n]}\times S\times S^{[n]}$ onto the product of the $i$-th and $j$-th factors.
The relative extension sheaf
\begin{equation}
\label{eq-E}
E := \SheafExt^1_{\pi_{13}}\left(\pi_{12}^*\U,\pi_{23}^*\U\right)
\end{equation}
is a  rank $2n-2$ reflexive  sheaf  over $S^{[n]}\times S^{[n]}$
\cite[Prop. 4.1]{markman-hodge}. The sheaf $E$ is infinitesimally rigid, i.e., $\Ext^1(E,E)$ vanishes, by
\cite[Lemma 5.2]{generalized-deformations}. We prove in this note the following statement.

\begin{thm}
\label{cor-stability-of-E-hilbert-scheme-case}
\begin{enumerate}
\item
\label{cor-item-stability-for-every-omega}
When $\Pic(S)$ is trivial the sheaf 
$E$ is 
slope-stable with respect to every K\"{a}hler class.
\item
\label{cor-item-stability-for-some-omega}
If $0<\rank(\Pic(S))\leq 19$,
then $E$ is $\omega\boxplus \omega$-slope-stable with respect to some 
K\"{a}hler class $\omega$ on $S^{[n]}$.
\end{enumerate}
\end{thm}

The K\"{a}hler class $\omega\boxplus \omega$ above is the sum of the pullbacks of $\omega$ via the two projections.
The theorem is proven in Section \ref{sec-restriction-to-the-diagonal}. 
The class $c_2(\SheafEnd(E))$ remains of Hodge type $(2,2)$ on $X\times X$, for all K\"{a}hler manifolds $X$ deformation equivalent to $S^{[n]}$, by \cite[Lemma 3.2 and Prop. 3.4]{markman-hodge}. 
Such a manifold $X$ is called of {\em $K3^{[n]}$-type.}
The slope stability result of the above theorem implies that $\SheafEnd(E)$ deforms to a reflexive sheaf $\A$ of  Azumaya algebras over $X\times X$, for all manifolds $X$ of $K3^{[n]}$-type, by a theorem of Verbitsky 
\cite[Theorem 3.19]{kaledin-verbitski-book} (see also \cite[Cor. 6.11 and Prop. 6.16]{markman-hodge}).
Slope stability was proven earlier in \cite[Theorem 7.4]{markman-hodge} for the analogue of the sheaf $E$ over the cartesian square $M\times M$ of a moduli space $M$ 
of stable coherent sheaves of rank $2n-2$ on a projective $K3$ surface $S$. There
$M$ was chosen so that the sheaf $E$ is twisted by a Brauer class of maximal order equal to the rank $2n-2$ of $E$, a fact that implies the slope-stability of $E$, as well as the slope polystability of the untwisted $\SheafEnd(E)$, with respect to every K\"{a}hler class 
\cite[Prop. 6.5]{markman-hodge}. 

Theorem \ref{cor-stability-of-E-hilbert-scheme-case} plays an important role in two joint works with S. Mehrotra
\cite{torelli,generalized-deformations}. There we need the stability of $E$ in Theorem \ref{cor-stability-of-E-hilbert-scheme-case} 
(rather than over a  moduli space $M$ of higher rank sheaves), since only when the moduli space is a Hilbert scheme we could prove that the integral transform using the universal sheaf from the derived category of the $K3$ surface to that of the moduli space is a $\PP^n$-functor 
\cite[Theorem 1.1]{generalized-deformations} (see also \cite{addington}).
Theorem \ref{cor-stability-of-E-hilbert-scheme-case}  is used in the 
construction of a $21$-dimensional moduli space $\fM$ of isomorphism classes of triples $(X,\eta,\A)$, where 
$X$ is a manifold of $K3^{[n]}$-type, 
$\eta:H^2(X,\Integers)\rightarrow\Lambda_{n}$ is an isometry with a fixed lattice $\Lambda_{n}$, and
$\A$ is a slope stable reflexive sheaf of Azumaya algebras over $X\times X$   \cite{torelli}. 
We prove a Torelli theorem for the pairs $(X,\A)$ deformation equivalent to $(S^{[n]},\SheafEnd(E))$, stating that if 
$(X',\A')$ is another such pair, and $\A$ and $\A'$ are both $\omega\boxplus\omega$-slope-stable, with respect to the same K\"{a}hler class $\omega$ on $X$, then $\A'$ is isomorphic to $\A$ or $\A^*$ (the latter is the same sheaf as $\A$, but with the dual multiplication) \cite[Theorem 1.11]{torelli}.
In \cite{generalized-deformations} we associate to a triple $(X,\eta,\A)$ in $\fM$ a pre-triangulated $K3$-category, resulting in a global generalized (non-commutative and gerby) deformation of the derived categories of coherent sheaves on $K3$ surfaces (which are associated to the triples $(S^{[n]},\eta,\SheafEnd(E))$ by the construction). 


\medskip
{\bf Acknowledgments:} This work was 
partially supported by NSA grant H98230-13-1-0239. I would like to thank the referee for his comments, which helped improve the exposition.
%
\section{Density}
Let $\Lambda$ be the $K3$ lattice, namely the unique even unimodular lattice of rank $22$ and signature $(3,19)$. 
A marked $K3$ surface $(S,\eta)$ consists of 
a $K3$ surface $S$ and an isometry $\eta:H^2(S,\Integers)\rightarrow \Lambda$.
Choose one of the two connected components  of the moduli space of isomorphism classes of marked  $K3$ surfaces, not necessarily projective, and denote it by $\fM_{K3}$.
Let $\Omega_{K3}$ be the corresponding period domain, and let
$P:\fM_{K3}\rightarrow \Omega_{K3}$ be the period map \cite{looijenga-peters}. 
The map $P$ is a local homeomorphism.

The signed isometry group $O^+(\Lambda)$ acts on $\fM_{K3}$, by 
$g(S,\eta)=(S,g\eta)$, and on $\Omega_{K3}$, and the period map is $O^+(\Lambda)$-equivariant.
Over $\fM_{K3}$ we have a universal family of $K3$ surfaces $\S\rightarrow \fM_{K3}$, since the automorphism group of a $K3$ surface acts faithfully on its degree $2$ cohomology. We get over $\fM_{K3}$ 
a universal Douady space $\S^{[n]}\rightarrow  \fM_{K3}$, as well as the fiber square of the latter 
$\S^{[n]}\times_{\fM_{K3}}\S^{[n]}\rightarrow \fM_{K3}$. 
Let $\overline{\U}$ be the ideal sheaf of the universal subscheme of $\S\times_{\fM_{K3}}\S^{[n]}$ 
and let $\pi_{ij}$ be the projection from
$\S^{[n]}\times_{\fM_{K3}}\S\times_{\fM_{K3}}\S^{[n]}$ onto the fiber product of the $i$-th and $j$-th factors.
The universal 
version $\E$ over $\S^{[n]}\times_{\fM_{K3}}\S^{[n]}$ of the rank $2n-2$ sheaf $E$
is the relative extension sheaf
\begin{equation}
\label{eq-equivariant-E}
\E \:= \ \SheafExt^1_{\pi_{13}}(\pi_{12}^*\overline{\U},\pi_{23}^*\overline{\U}).
\end{equation}
The universal family $\S\rightarrow \fM_{K3}$ is $O^+(\Lambda)$-equivariant, by the universal property of the universal family.
Hence, so are the universal Douady space and the universal subscheme. The sheaves $\overline{\U}$ 
and $\E$ are thus $O^+(\Lambda)$-equivariant as well. The following is a consequence of a density theorem of Verbitsky.

\begin{lem}
\label{lemma-W-contains-marked-pairs-of-non-maximal-Picard-rank}
Let $W\subset \fM_{K3}$ be a non-empty open $O^+(\Lambda)$-invariant subset.
Then $W$ contains every marked pair $(S,\eta)$ such that 
the rank of the Picard group of $S$ is $\leq 19$.
\end{lem}

\begin{proof}
The image $P(W)$ is an open $O^+(\Lambda)$-invariant subset of $\Omega_{K3}$. 
The stabilizer of $P(S,\eta)$ in $O^+(\Lambda)$ acts transitively on the fiber of $P$ over $P(S,\eta)$, 
by the Global Torelli Theorem \cite{burns-rapoport}
(see also \cite[Lemma 10.4]{looijenga-peters}). Hence, $W=P^{-1}(P(W))$. 
If the rank of $\Pic(S)$ is less than or equal to $19$, then the $O^+(\Lambda)$-orbit of $P(S,\eta)$ 
is dense in $\Omega_{K3}$, by \cite{verbitsky-ergodic},
so it  intersects $P(W)$ and is thus contained in $P(W)$. 
Hence, the $O^+(\Lambda)$-orbit of $(S,\eta)$ is contained in $W$.
\end{proof}
%
\section{A canonical subsheaf}
\label{sec-a-canonical-subsheaf}

A coherent sheaf $F$ on a complex manifold of dimension $d$ is said to be {\em pure of codimension} $c$
if the support of every subsheaf of $F$ has codimension $c$.
\begin{lem}
\label{lemma-reflexive-criterion}
\cite[Cor. 1.5]{hartshorne-reflexive}
Let $0\rightarrow F'\rightarrow F\rightarrow Q\rightarrow 0$ be an exact sequence of coherent sheaves on a complex manifold.
\begin{enumerate}
\item
\label{lemma-item-iff}
Assume that $F$ is reflexive. Then $F'$ is reflexive, if and only if either the torsion subsheaf of $Q$ is pure of codimension $1$ or $Q$ is torsion free.
\item
\label{lemma-item-torsion-free-F}
If $F$ is torsion free and $F'$ is reflexive, then either $Q$ is torsion free, or the torsion subsheaf of $Q$ is pure of codimension $1$.
\end{enumerate}
\end{lem}

\begin{proof}
Part (\ref{lemma-item-iff}) is proven in \cite[Cor. 1.5]{hartshorne-reflexive}. 
Part (\ref{lemma-item-torsion-free-F}): The torsion subsheaf of $F^{**}/F'$ either vanishes, or is pure of codimension $1$,
by Part (\ref{lemma-item-iff}), and the composition $Q\cong F/F'\rightarrow F^{**}/F'$ is injective.
\end{proof}

Let $\iota:\Z\rightarrow S\times S^{[n]}$ be the inclusion of the universal subscheme and let $I_\Z$ be its ideal sheaf.
Let $q_i$ be the projection from $S\times S^{[n]}$ to the $i$-th factor, $i=1,2$.
We get the split short exact sequence of locally free sheaves
\begin{equation}
\label{eq-u}
0\rightarrow \StructureSheaf{S^{[n]}}\RightArrowOf{u} q_{2,*}\iota_*\StructureSheaf{\Z} \rightarrow A_0\rightarrow 0.
\end{equation}
We have $c_1(A_0)=c_1(q_{2,*}\iota_*\StructureSheaf{\Z})=-\delta$, where $2\delta\in H^2(S^{[n]},\Integers)$ is the class of the divisor of non-reduced subschemes \cite[Sec. 5]{ELG}.
Let $\I\subset S^{[n]}\times S^{[n]}$ be the incidence subvariety consisting of pairs $(Z_1,Z_2)$ of length $n$ subschemes,
which are not disjoint, $Z_1\cap Z_2\neq \emptyset$.
Let $p_i$ be the projection from $S^{[n]}\times S^{[n]}$ to the $i$-th factor, $i=1,2$. 

\begin{prop}
\label{prop-canonical-subsheaf-of-E}
The sheaf $E$, given in (\ref{eq-E}), fits in the left exact sequence 
\begin{equation}
\label{eq-canonical-subsheaf}
0\rightarrow p_2^*A_0\RightArrowOf{h} E \RightArrowOf{j} p_1^*A_0^*
\end{equation}
and the co-kernel of $j$ is supported, set theoretically, on $\I$.
\end{prop}

\begin{proof}
Set $\StructureSheaf{}:=\StructureSheaf{S^{[n]}\times S\times S^{[n]}}$.
Apply the functor $R\SheafHom(\pi_{12}^*I_\Z,\bullet)$ to the short exact sequence
\[
0\rightarrow \pi_{23}^*I_\Z\rightarrow \StructureSheaf{}\rightarrow \pi_{23}^*\iota_*\StructureSheaf{\Z}\rightarrow 0
\]
to get the exact triangle
\[
R\SheafHom(\pi_{12}^*I_\Z,\pi_{23}^*I_\Z)\rightarrow 
R\SheafHom(\pi_{12}^*I_\Z,\StructureSheaf{})\rightarrow 
R\SheafHom(\pi_{12}^*I_\Z,\pi_{23}^*\iota_*\StructureSheaf{\Z}).
\]
Note the isomorphism $\SheafHom(\pi_{12}^*I_\Z,\StructureSheaf{})\cong \StructureSheaf{}$.
Applying the functor $R\pi_{13,*}$ and taking sheaf cohomology of the resulting exact triangle we get the long exact sequence
\begin{eqnarray}
\label{eq-long-exact-sequence-of-canonical-subsheaf}
0&\rightarrow& \StructureSheaf{S^{[n]}\times S^{[n]}}\rightarrow
\pi_{13,*}\SheafHom(\pi_{12}^*I_\Z,\pi_{23}^*\iota_*\StructureSheaf{\Z})
\RightArrowOf{\tilde{h}}
E
\RightArrowOf{j}
\SheafExt^1_{\pi_{13}}(\pi_{12}^* I_\Z,\StructureSheaf{})\rightarrow 
\\
\nonumber
&&
\SheafExt^1_{\pi_{13}}(\pi_{12}^*I_\Z,\pi_{23}^*\iota_*\StructureSheaf{\Z})
\RightArrowOf{k} \SheafExt^2_{\pi_{13}}(\pi_{12}^*I_\Z,\pi_{23}^*I_\Z)\rightarrow \cdots
\end{eqnarray}
Away from $\I$  the natural homomorphism
\[
\pi_{13,*}\SheafHom(\StructureSheaf{},\pi_{23}^*\iota_*\StructureSheaf{\Z})
\rightarrow
\pi_{13,*}\SheafHom(\pi_{12}^*I_\Z,\pi_{23}^*\iota_*\StructureSheaf{\Z})
\]
is an isomorphism, 
and the left hand sheaf is naturally isomorphic to $p_2^*q_{2,*}\iota_*\StructureSheaf{\Z}$.
Composing the above displayed homomorphism with $\tilde{h}$ 
we get the injective homomorphism $h$ in Equation (\ref{eq-canonical-subsheaf}).
The exactness of the sequence (\ref{eq-long-exact-sequence-of-canonical-subsheaf}) 
implies that the image of $h$ is saturated, away from $\I$.
The image of $h$ must be a saturated subsheaf of $E$ everywhere, by Lemma \ref{lemma-reflexive-criterion} (\ref{lemma-item-iff}), 
since the image is locally free and the codimension of $\I$ is $2$.

The relative extension sheaves 
$\SheafExt^1_{\pi_{13,*}}(\StructureSheaf{},\StructureSheaf{})$ and
$\SheafExt^2_{\pi_{13,*}}(\pi_{12}^*I_\Z,\StructureSheaf{})$ both vanish.
Hence, we get the short exact sequence
\[
0\rightarrow \SheafExt^1_{\pi_{13}}(\pi_{12}^*I_\Z,\StructureSheaf{})\rightarrow
\SheafExt^2_{\pi_{13}}(\pi_{12}^*\iota_*\StructureSheaf{\Z},\StructureSheaf{})\RightArrowOf{\tilde{u}^*}
\SheafExt^2_{\pi_{13}}(\StructureSheaf{},\StructureSheaf{})\rightarrow 0.
\]
Grothendieck-Verdier Duality, combined with the triviality of the relative canonical line bundle $\omega_{\pi_{13}}$, 
identifies the pullback $p_1^*(u^*)$  of the 
dual of the homomorphism $u$ in Equation (\ref{eq-u})
with the of the homomorphism $\tilde{u}^*$ above. Hence, $\SheafExt^1_{\pi_{13}}(\pi_{12}^*I_\Z,\StructureSheaf{})$ is isomorphic to
the kernel $p_1^*A_0^*$ of $p_1^*(u^*)$. Using the latter isomorphism we obtain the homomorphism $j$ in Equation
(\ref{eq-canonical-subsheaf}) from the homomorphism $j$ in the long exact sequence 
(\ref{eq-long-exact-sequence-of-canonical-subsheaf}).
The co-kernel of $j$ is equal to the kernel of $k$ in (\ref{eq-long-exact-sequence-of-canonical-subsheaf}) from the sheaf 
$\SheafExt^1_{\pi_{13}}(\pi_{12}^*I_\Z,\pi_{23}^*\iota_*\StructureSheaf{\Z})$ to 
$\SheafExt^2_{\pi_{13}}(\pi_{12}^*I_\Z,\pi_{23}^*I_\Z)$.
The former is isomorphic to $\SheafExt^2_{\pi_{13}}(\pi_{12}^*\iota_*\StructureSheaf{\Z},\pi_{23}^*\iota_*\StructureSheaf{\Z})$ and is
thus supported set theoretically on $\I$ and the latter is supported on the diagonal.
Hence, the co-kernel of $j$ is supported on $\I$.
\end{proof}

%
\section{Blow-up of the diagonal and restriction to the exceptional divisor}
\label{sec-restriction-to-the-diagonal}
We recall  one more crucial property of the relative extension sheaf $E$ given in (\ref{eq-E}). 
Let $\beta:Y\rightarrow S^{[n]}\times S^{[n]}$ be the blow-up of $S^{[n]}\times S^{[n]}$
along the diagonal $\Delta$. Denote by $\widetilde{\Delta}$ the exceptional divisor in $Y$. 
Note that $\widetilde{\Delta}$ is naturally isomorphic to $\PP(T^*\Delta)$.
Set
\[
V \ := \ \beta^*{E}(\widetilde{\Delta})/{\rm torsion}.
\]
Let $\pi:\widetilde{\Delta}\rightarrow \Delta$ be the restriction of $\beta$.
Let $\ell\subset \pi^*T^*\Delta$ be the tautological line sub-bundle. The
restriction of $\StructureSheaf{\widetilde{\Delta}}(\widetilde{\Delta})$ to $\widetilde{\Delta}$
is isomorphic to $\ell$. Let $\ell^\perp$ be the sub-bundle  of $\pi^*T^*\Delta$ 
symplectic-othogonal to $\ell$. Note that the symplectic structure on $T^*\Delta$ induces one on
$\ell^\perp/\ell$.

\begin{lem}
\label{lemma-simplicity}
The vector space $\End(\ell^\perp/\ell)$ is one-dimensional.
\end{lem}

\begin{proof}
The restriction of $\ell^\perp/\ell$ to each fiber of $\pi$ is simple, by \cite[Lemma 7.3 (2)]{torelli}.
Hence, the sheaf $\pi_*\SheafEnd(\ell^\perp/\ell)$ is the trivial line-bundle over $S^{[n]}$.
The statement follows from the isomorphism
$H^0(\widetilde{\Delta},\SheafEnd(\ell^\perp/\ell))=H^0(S^{[n]},\pi_*\SheafEnd(\ell^\perp/\ell))$.
\end{proof}

\begin{prop}
\label{prop-V}
$V$ is locally free. The restriction of $V$ to $\widetilde{\Delta}$ is isomorphic to $\ell^\perp/\ell$. 
\end{prop}

\begin{proof}
When the $K3$ surface $S$ is projective, the statement is precisely 
\cite[Prop. 4.1 parts (3) and (6)]{markman-hodge}.
The proof provided there uses a global complex over $S^{[n]}\times S^{[n]}$ of locally free sheaves,
representing the object $R\pi_{13}\left(R\SheafHom(\pi_{12}^*\U,\pi_{23}^*\U)\right)$.
However, the argument provided there is local and goes through when such a complex of  locally free sheaves
is given only in a complex analytic neighborhood of a point of the diagonal in $S^{[n]}\times S^{[n]}$.
Hence, the statement that $V$ is locally free holds without the assumption that $S$ is projective. 
It remains to be proved that the restriction of $V$ to $\widetilde{\Delta}$ is isomorphic to $\ell^\perp/\ell$
even when $S$ is non-projective.

Let $\beta:\Y\rightarrow \S^{[n]}\times_{\fM_{K3}}\S^{[n]}$ be the blow-up of the diagonal and 
$\D\subset \Y$ the exceptional divisor. 
The sheaf $\V:=(\beta^*\E)(\D)/\mbox{torsion}$ is locally free, by the above argument.
Let $\phi:\D\rightarrow \fM_{K3}$ and $\psi:\S^{[n]}\rightarrow \fM_{K3}$
be the natural morphisms. 
$\D$ is naturally isomoprhic to $\PP(T_\psi)$. 
Let $\LB$ be the tautological subbundle of $\phi^*T_\psi$ over $\D$. 
The sheaf $\psi_*\Omega^2_\psi$ is a line-bundle over $\fM_{K3}$. We have a natural injective homomorphism
$
\psi^*\psi_*\Omega^2_\psi\rightarrow \Omega^2_\psi.
$
We get a well defined symplectic-orthogonal subbundle $\LB^\perp\subset \phi^*T\psi$
as well as the quotient $\LB^\perp/\LB$.

The fiber of the sheaf 
$R^{4n-1}\phi_*\left((\LB^\perp/\LB)^*\otimes \restricted{\V}{\D}\otimes \omega_\phi\right)$
at every marked pair $(S,\eta)$ maps isomorphically onto the vector space
$\Ext^{4n-1}(\ell^\perp/\ell,\restricted{V}{\widetilde{\Delta}}\otimes\omega_{\widetilde{\Delta}})$,
by the Base-Change Theorem. The latter vector space 
is isomorphic to $\Hom(\restricted{V}{\widetilde{\Delta}},\ell^\perp/\ell)^*$ and 
is thus one-dimensional whenever $S$ is projective, by \cite[Prop. 4.1 parts (3) and (6)]{markman-hodge} and 
Lemma \ref{lemma-simplicity}. The locus of projective $K3$ surfaces is dense in $\fM_{K3}$. 
Hence, $\Hom(\restricted{V}{\widetilde{\Delta}},\ell^\perp/\ell)$
is one-dimensional over a non-empty Zariski open subset of $\fM_{K3}$, by the semi-continuity theorem.

Considering the sheaf 
$R^{4n-1}\phi_*\left(\restricted{\V^*}{\D}\otimes \LB^\perp/\LB \otimes  \omega_\phi\right)$
we conclude similarly that $\Hom(\ell^\perp/\ell,\restricted{V}{\widetilde{\Delta}})$
is one-dimensional over a non-empty Zariski open subset of $\fM_{K3}$.
Hence, the sheaves
\begin{eqnarray*}
L_1 & := & \phi_*\SheafHom(\LB^\perp/\LB,\restricted{\V}{\D}) \ \mbox{and}
\\
L_2 & := & \phi_*\SheafHom(\restricted{\V}{\D},\LB^\perp/\LB)
\end{eqnarray*}
are both locally free of rank $1$ over a non-empty Zariski open subset $U'$ of $\fM_{K3}$.
The fiber of $L_i$, $i=1,2$, is spanned by an isomorphism over every projective marked $K3$ surface.
Hence, the composition maps
$L_1\otimes L_2\rightarrow \phi_*\SheafEnd(\LB^\perp/\LB)$ and
$L_1\otimes L_2\rightarrow \phi_*\SheafEnd(\restricted{\V}{\D})$
are isomorphisms over a non-empty Zariski open subset $U$ of $U'$.
We conclude that $\restricted{V}{\widetilde{\Delta}}$ is isomorphic to $\ell^\perp/\ell$ for every marked pair $(S,\eta)$ in $U$.
The set $U$ is $O^+(\Lambda)$-invariant and thus contains every marked pair $(S,\eta)$ with Picard rank
$\leq 19$, by Lemma \ref{lemma-W-contains-marked-pairs-of-non-maximal-Picard-rank}.
Such is the Picard rank of every non-projective $K3$ surface.
\end{proof}

\begin{prop}
\label{prop-ell-perp-mod-ell-is-stable}
When $\Pic(S)$ is trivial 
the vector bundle $\ell^\perp/\ell$ has a unique non trivial subsheaf $C$ of rank less than $2n-2$. The rank of $C$ is $n-1$.
\end{prop}

Proposition \ref{prop-ell-perp-mod-ell-is-stable}
will be proven in Section \ref{sec-stability}.

\begin{proof}[Proof of Theorem \ref{cor-stability-of-E-hilbert-scheme-case}]
\ref{cor-item-stability-for-every-omega})
We prove first that every non-trivial proper subsheaf of $V$ has rank $n-1$.
Let $F$ be a non-trivial subsheaf of $V$ of lower rank. We may assume that $F$ is saturated in $V$. 
Then $F$ is a subbundle of $V$ away from the locus where $V/F$ is not locally free. 
That locus has codimension at least $2$, since 
$V/F$ is torsion free.
Thus, $F$ restricts to a subsheaf 
of the restriction of $V$ to $\widetilde{\Delta}$ of the same rank as $F$. 
The restriction of $V$  to $\widetilde{\Delta}$ is isomorphic to 
$\ell^\perp/\ell$, by Proposition \ref{prop-V}.
We conclude that $\rank(F)=n-1$, by Proposition \ref{prop-ell-perp-mod-ell-is-stable}. 

Let $F$ be a non-trivial proper saturated subsheaf of $E$. Then $F$ is reflexive, being saturated in the reflexive sheaf $E$, and
$\rank(F)=n-1$.
In particular, the image $F_0$ of the homomorphism $h$ in the sequence  
(\ref{eq-canonical-subsheaf}) does not have any non-trivial subsheaf of lower rank. Furthermore, either $F=F_0$, or 
$F\cap F_0=0$.

Assume that $F\cap F_0=0$. Composing the inclusion $F\rightarrow E$ with the homomorpism $j$ in the sequence 
(\ref{eq-canonical-subsheaf}) we get the injective homomorphism $g:F\rightarrow p_1^*A_0^*$. 
The sheaf $F$ is reflexive, $p_1^*A_0^*$ is locally free, and the rank of both is $n-1$. 
Hence, every irreducible component of the support of the co-kernel of $g$ must be of codimension $1$,
by Lemma \ref{lemma-reflexive-criterion} (\ref{lemma-item-iff}).
The co-kernel of $g$ surjects onto the co-kernel of the homomorphism $j$, and the latter is supported, set theoretically,  on the
codimension two subvariety $\I$, by Proposition \ref{prop-canonical-subsheaf-of-E}. 
Hence, $\I$ is contained in some effective divisor of $S^{[n]}\times S^{[n]}$. But such a divisor does not exists, since the
only effective divisors on $S^{[n]}$ are multiples of the divisor of non-reduced subschemes. A contradiction. Hence, $F=F_0$.

We have the equalities $c_1(F_0)=-p_2^*c_1(A_0)$, $c_1(E/F_0)=p_1^*c_1(A_0^*)$, and 
$c_1(A_0)=-\delta$, where $2\delta$ is an effective divisor, as noted in Section \ref{sec-a-canonical-subsheaf}. 
Hence, the unique non-trivial subsheaf $F_0$ of $E$ does not slope-destabilize $E$ with respect to any K\"{a}hler class on $S^{[n]}\times S^{[n]}$.

\ref{cor-item-stability-for-some-omega})
Let $W\subset \fM_{K3}$ be the subset consisting of pairs $(S,\eta)$ for which the sheaf $\E$, 
given in (\ref{eq-equivariant-E}),
restricts to $S^{[n]}\times S^{[n]}$
as an $\omega\boxplus \omega$-slope-stable sheaf with respect to some K\"{a}hler class $\omega$ on $S^{[n]}$.
$W$ is an open subset, since every K\"{a}hler class extends to a section of K\"{a}hler classes for the universal Douady space over
an open subset of $\fM_{K3}$ (see \cite[Th. 9.3.3]{voisin-book-vol1}) and since
slope-stability is an open condition. 
$W$ is clearly $O^+(\Lambda)$-invariant.
Hence, $W$ contains every marked pair $(S,\eta)$ in $\fM_{K3}$, with $\rank(\Pic(S))\leq 19$, by
Lemma \ref{lemma-W-contains-marked-pairs-of-non-maximal-Picard-rank}.
\end{proof}

\hide{
%
\section{Saturated restrictions to codimension two subvarieties}

Let $M$ be a smooth variety, $V$ a locally free sheaf  over $M$, $F$ a saturated 
subsheaf of $V$, and $Z$ a reduced subscheme  of pure codimension $2$ in $M$.
Denote the restriction of $V$ to $Z$ by $\restricted{V}{Z}$. 

\begin{defi}
The {\em saturated restriction} of $F$ to $Z$ is the saturation in $\restricted{V}{Z}$
of the image of the restriction of $F$ to $Z$. 
\end{defi}

Note that $F$ is necessarily reflexive, hence locally free in a neighborhood of 
a generic point of each irreducible component of $Z$. In general, 
the rank of the saturated restriction of $F$ need not be equal to that of $F$, as the following example illustrates. 

\begin{example}
Let $s$ be a section of a rank $2$ vector bundle $W$ over $M$ vanishing along a smooth subvariety $Z$ of codimension $2$ in $M$.
We get the short exact sequence
\[
0\rightarrow \wedge^2W^*\RightArrowOf{s} W^* \RightArrowOf{s} I_Z\rightarrow 0.
\]
The restriction of $W^*$ to $Z$ maps isomorphically onto the restriction of $I_Z$ to $Z$, and both are isomorphic to the co-normal bundle of $Z$. Hence, the saturated restriction to $Z$ of the image of  $\wedge^2W^*$ in $W^*$ vanishes.
\end{example}

Following is a criterion for the rank of the saturated restriction to remain the same. Assume that $Z$ is a smooth subvariety.
Denote by $\zeta:Z\rightarrow M$ the inclusion. Let $\beta:\hat{M}\rightarrow M$ be the blow-up of $M$ centered at $Z$, let
$d:D\rightarrow \hat{M}$ be the inclusion of the exceptional divisor, and let $p:D\rightarrow Z$ be the natural projection.
Denote by $G$ the saturation of the image of $\beta^*F$ in $\beta^*V$ and let $G_D$ be the image of $d^*G$  in $d^*\beta^*V$.

\begin{lem}
Assume that there exists a saturated subsheaf $G_Z$ of $\zeta^*V$, such that $G_D=p^*G_Z$. 
Then $G_Z$ is the saturated restriction of $F$ to $Z$ and $\rank(F)=\rank(G_Z)$.
\end{lem}

\begin{proof}
The equality $\rank(G)=\rank(G_D)$ holds, since $D$ is a divisor, and $\rank(G)=\rank(F)$. The equality $G_D=p^*G_Z$
thus implies that $\rank(F)=\rank(G_Z)$.

We prove next that $G_Z$ is the saturated restriction of $F$. We may assume that $G_Z$ is a subbundle of $\zeta^*V$,
possibly after replacing $M$ by the complement in $M$ of a closed subset of $Z$ of codimension $\geq 2$ in $Z$.
Then $G_D$ is a subbundle of $d^*\beta^*V$. We may further assume that $G$ is a subbundle of $\beta^*V$, 
possibly after replacing $M$ by the complement of a closed subset of codimension $\geq 2$ disjoint from $Z$.
Set $r:=\rank(F)$. 
Let $\pi:Gr(r,V)\rightarrow M$ be the Grassmannian bundle of $r$-dimensional subspaces of the fibers of $V$.
Let
\[
0\rightarrow \tau \rightarrow \pi^*V\rightarrow Q\rightarrow 0
\]
be the exact sequence of the tautological sub and quotient bundles. 
The subbundle $G$ of $\beta^*V$ determines a morphism $g:\hat{M}\rightarrow Gr(r,V)$, such that $\beta=\pi\circ g$ and
$g^*\tau=G$. 
\[
\xymatrix{
D \ar[r]^d\ar[d]_{p} & \hat{M}\ar[r]^{g}\ar[d]_\beta & Gr(r,V)\ar[dl]^{\pi}
\\
Z \ar[r]_\zeta & M
}
\]
We get the short exact sequence over $\hat{M}$
\[
0\rightarrow g^*\tau \rightarrow \beta^*V\rightarrow g^*Q\rightarrow 0.
\]
The sheaves $R^i\beta_*(g^*\tau)$ vanish, for $i>0$,  by the equality $d^*\tau=p^*G_Z$ and 
\cite[Cor. 12.9]{hartshorne}. Similarly, $R^i\beta_*\beta^*V$ vanishes, for $i>0$. We conclude that $R^i\beta_*g^*Q$ vanishes, for $i>0$,
and get the short exact sequence
\[
0\rightarrow \beta_*g^*\tau\rightarrow V\rightarrow \beta_*g^*Q\rightarrow 0.
\]
The equality $d^*\tau=p^*G_Z$ implies that the spaces of global sections of the restriction of $g^*Q$ to every fiber of $\beta$ are all
of the same dimension equal to the rank of $Q$. Hence, 
the sheaf $\beta_*g^*Q$ is locally free, by 
\cite[Cor. 12.9]{hartshorne}. Consequently, the sheaf $\beta_*g^*\tau$ is a subbundle of $V$.
Now $F$ is equal to $\beta_*g^*\tau$ away from $Z$ and both are saturated subsheaves. Hence $F$ is a subbundle of $V$ and its 
restriction to $Z$ is $G_Z$.
\end{proof}

}

%
\section{Proof of Proposition \ref{prop-ell-perp-mod-ell-is-stable}}
\label{sec-stability}

Let $S$ be a $K3$ surface with a trivial Picard group.

\begin{lem}
\label{lemma-vanishing-of-global-sections-of-symmetric-powers}
\begin{enumerate}
\item
\label{lemma-item-vanishing-of-global-sections-of-symmetric-powers}
$H^0(S^n,\Sym^k(T^*S^n))=0,$ for all $n>0$ and $k>0.$
\item
\label{lemma-item-projective-bundle-does-not-contain-any-effective-divisor}
$\PP(T^*S^n)$ does not contain any effective divisors.
\end{enumerate}
\end{lem}

\begin{proof}
\ref{lemma-item-vanishing-of-global-sections-of-symmetric-powers}) 
The vector bundles $\Sym^d(T^*S)$ are stable, as the holonomy of $T^*S$ is $Sp(2)\cong SL(2)$
and the $d$-th symmetric power of the standard rank $2$
representation of $SL(2)$ is irreducible, for all $d\geq 0$. 
The spaces $H^0(S,\Sym^d(T^*S))$ thus vanish for $d>0$, since $c_1(\Sym^d(T^*S))=0$.
Now $H^0(S^n,\Sym^k(T^*S^n))$ is the direct sum, over all ordered
partitions $k=d_1+ \dots +d_n$, of the tensor product 
$\otimes_{i=1}^nH^0(S,\Sym^{d_i}(T^*S))$. The latter tensor product vanishes if $k>0$,
since at least one $d_i$ is positive.

\ref{lemma-item-projective-bundle-does-not-contain-any-effective-divisor})
Any line bundle is a tensor power of the tautological line bundle
$\StructureSheaf{\PP(T^*S^n)}(1)$. The statement follows immediately from part
\ref{lemma-item-vanishing-of-global-sections-of-symmetric-powers}.
\end{proof}

\begin{lem}
\label{lemma-saturated-subsheaves-of-the-trivial-sheaf}
Let $X$ be a compact complex manifold, which does not have any effective divisors, and $L$ a line bundle on $X$. Let $V$ be a finite dimensional vector space. Then every saturated subsheaf of $V\otimes_\ComplexNumbers L$ is of the form 
$W\otimes_\ComplexNumbers L$, for some subspace $W$ of $V$.
\end{lem}

\begin{proof}
Let $F$ be a saturated subsheaf of $V\otimes_\ComplexNumbers L$ and $\iota:F\rightarrow V\otimes_\ComplexNumbers L$
the inclusion homomorphism. Choose a generic quotient space $Q$ of $V$, such that the composite homomorphism $h$, given by
$F\rightarrow V\otimes_\ComplexNumbers L\rightarrow Q\otimes_\ComplexNumbers L$, is injective. The sheaf $F$ is reflexive, being a saturated subsheaf of a locally free sheaf. Thus, 
the cokernel of $h$ either vanishes, or it is supported on a codimension $1$ subscheme, by Lemma \ref{lemma-reflexive-criterion} (\ref{lemma-item-iff}).
The latter case is excluded by the assumption that $X$ does not have any effective divisors. Hence $h$ is an isomorphism. Let
$\iota_0:Q\rightarrow V$ be the composition 
$
Q\cong \Hom(L,F) \RightArrowOf{\iota_*}\Hom(L,V\otimes_\ComplexNumbers L)\cong V.
$
Then $F$ is the image of $\iota_0\otimes 1:Q\otimes_\ComplexNumbers L\rightarrow V\otimes_\ComplexNumbers L$.
\end{proof}

\begin{lem}
\label{lemma-on-equivariant-subsheaves}
Let $X$, $L$, and $V$ be as in Lemma \ref{lemma-saturated-subsheaves-of-the-trivial-sheaf}.
Let $G$ be a finite group of automorphisms of $X$. Assume that $L$ is endowed with a $G$-equivariant structure,
$V$ is an irreducible $G$-representation, and endow $V\otimes_\ComplexNumbers L$ with the associated $G$-equivariant structure. 
Then $V\otimes_\ComplexNumbers L$ does not have any non-trivial saturated $G$-equivariant subsheaf of lower rank.
\end{lem}

\begin{proof}
Let $F\subset V\otimes_\ComplexNumbers L$ be a saturated subsheaf. 
Then $F=W\otimes_\ComplexNumbers L$, for some subspace $W$ of $V$, by 
Lemma \ref{lemma-saturated-subsheaves-of-the-trivial-sheaf}. $G$-equivariance of $F$ implies that
$Hom(L,F)$ is a $G$-subrepresentation of the irreducible representation $V$. Hence, $W=0$ or $W=V$.
\end{proof}

\hide{
\begin{lem}
Let $S$ be a $K3$ surface with trivial Picard group. Then $\PP[TS\oplus \StructureSheaf{S}]$ has a unique effective reduced divisor 
$\PP(TS)$.
\end{lem}

\begin{proof}
Let $D$ be an irreducible and reduced divisor in $X:=\PP[TS\oplus \StructureSheaf{S}]$. Then $D$ maps onto $S$, since $S$ does not contain any effective divisor. Thus, $D$ intersects the generic $\PP^2$ fiber of $\pi:\PP[TS\oplus \StructureSheaf{S}]\rightarrow S$
is an effective divisor of positive degree $d$. Let $\StructureSheaf{X}(-1)$ be the tautological line subbundle of 
$\pi^*(TS\oplus \StructureSheaf{S})$. Then $\StructureSheaf{X}(D)$ is isomorphic to $\StructureSheaf{X}(d)\otimes\pi^*L$, for some line bundle $L$ over $S$, which must be trivial, by the triviality of $\Pic(S)$. Hence,
\[
H^0(X,\StructureSheaf{X}(d))\cong H^0(S,\Sym^d[T^*S\oplus \StructureSheaf{S}])\cong 
H^0(S,\oplus_{k=0}^d\Sym^k(T^*S))\stackrel{{\rm Lemma} \ \ref{lemma-vanishing-of-global-sections-of-symmetric-powers}}{\cong}
H^0(S,\StructureSheaf{S})\cong \ComplexNumbers.
\]
We conclude that $\linsys{\StructureSheaf{X}(D)}=\linsys{\StructureSheaf{X}(d)}$ and the latter consists of the single effective divisor 
$\PP(TS)$. 
\end{proof}

\begin{lem}
Let $f:\X\rightarrow B$ be a smooth and proper morphism between smooth complex manifolds, and $D\subset \X$ an irreducible and reduced divisor, flat over $B$. Assume that for each point $b\in B$ the fiber $D_b$ of $D$ over $b$ is the unique reduced effective divisor in the fiber $X_b$ of $f$ over $b$ and $\StructureSheaf{X_b}(D_b)$ generates $\Pic(X_b)$. Let $V$ be a vector bundle of rank $n$ over $B$ and $F$ a saturated subsheaf of $f^*V$ of
rank $r$, such that $0<r<n$. Then $F$ is of the form $f^*W$, for a saturated subsheaf $W$ of $V$.
\end{lem}

\begin{proof}
Let $\iota:F\rightarrow f^*V$ be the inclusion homomorphism. 
Consider first the case $r=1$.
Then $F$ is isomorphic to $f^*L\otimes\StructureSheaf{\X}(-kD)$, for some integer $k$, and $\iota$ is a section of 
\[
H^0(\X,f^*[V\otimes L^{-1}]\otimes \StructureSheaf{\X}(kD))\cong 
H^0(B,[V\otimes L^{-1}]\otimes f_*\StructureSheaf{\X}(kD)).
\]
In particular, $k\geq 0$ and so $f_*\StructureSheaf{\X}(kD)$ is the trivial line-bundle, by our assumption on $D$. 
Let $\bar{\iota}:L\rightarrow V$ be the homomorphism associated to $\iota$ via the above displayed isomorphism.
Then $\iota$ factors through $f^*(\bar{\iota}):f^*L\rightarrow f^*V$ and so the image of $L$ in $V$ must be saturated, since $F$ is saturated in $f^*V$. Furthermore, we must have $k=0$ and $\iota=f^*\bar{\iota}$, since otherwise $F$ can not be a saturated subsheaf.

For $r\geq 1$ the above argument shows that the saturation of the image of $\wedge^r F$ in $f^*\wedge^rV$
is of the form $f^*L$, for some line-bundle $L$ and a homomorphism $\bar{\iota}:L\rightarrow \wedge^rV$ with a saturated image.
The subsheaf $F$ is a subbundle, away from a closed subscheme $\Z$ of $\X$ of co-dimension $\geq 2$ (supporting the singular locus of the torsion free quotient sheaf $f^*V/F$). 
Similarly, $L$ is a subbundle of $\wedge^rV$, away from a closed subscheme $\overline{\Z}$ of $B$ of co-dimension $\geq 2$.
We may assume that $L$ is a subbundle, possibly after replacing $B$ by $B\setminus \overline{\Z}$.
If $x$ is a point of $\X\setminus \Z$, then the fiber of $f^*L$ at $x$ is 
$\wedge^rF_x$. It follows that the morphism $B\rightarrow \PP(\wedge^rV)$ factors through the Pl\"{u}cker embedding of 
$Gr(r,V)$ in $\PP(\wedge^rV)$ and so $L$ is  of the form $\Wedge{r}W$ for a subbundle $W$ of $V$.
Hence, $F=f^*W$.
\end{proof}
}

\begin{proof}[Proof of Proposition \ref{prop-ell-perp-mod-ell-is-stable}]
Let $S^n$ be the $n$-th Cartesian product, $S^{(n)}$ the $n$-th symmetric
product, and $q:S^n\rightarrow S^{(n)}$ the quotient morphism. Denote by
$U\subset S^{(n)}$ the complement of the diagonal subscheme and set
$\widetilde{U}:=q^{-1}(U)$. Denote by $q:\widetilde{U}\rightarrow U$ the
covering map. 
We identify $U$ also as an open subset of the Hilbert scheme $S^{[n]}$.

Let $p_n:\PP(T^*S^n)\rightarrow S^n$ be the projection and denote by
$\tilde{\ell}$ the tautological line sub-bundle of $p_n^*T^*S^n$.
Denote by $\fS_n$ the symmetric group on $n$ letters.
Let $\sigma_n$ be a $\fS_n$-invariant symplectic structure on $S^n$.
Note that $\sigma_n$ is unique, up to a scalar factor. 
Denote by $\tilde{\ell}^\perp$ the sub-bundle of $p_n^*T^*S^n$
symplectic-orthogonal to $\tilde{\ell}$ with respect to $\sigma_n$. 
Both $\tilde{\ell}$ and $\tilde{\ell}^\perp$ are $\fS_n$-invariant sub-bundles of 
$p_n^*T^*S^n$. Hence, $\tilde{\ell}^\perp/\tilde{\ell}$ is endowed with the structure of an
$\fS_n$-equivariant vector bundle over $\PP(T^*S^n)$. 

Let $\pi_k:S^n\rightarrow S$ be the projection on the $k$-th factor.
Let $\tilde{\pi}_k:\PP(T^*S^n)\rightarrow S$ be the composition $\pi_k\circ p_n$.
Let $\tilde{\ell}_i$ be the projection of $\tilde{\ell}$ to $\tilde{\pi}_k^*T^*S$. 
The projection $\tilde{\ell}\rightarrow \tilde{\ell}_i$ is an isomorphism.
Let $\widetilde{C}$ be the  quotient
$(\oplus_{i=1}^n\tilde{\ell}_i)/\tilde{\ell}$. Then $\widetilde{C}$ is an $\fS_n$-invariant  subsheaf of $\tilde{\ell}^\perp/\tilde{\ell}$ of 
rank $n-1$. It thus corresponds to a saturated subsheaf $C$ of $\ell^\perp/\ell$ of rank $n-1$. 
Set $\widetilde{Q}:=[\tilde{\ell}^\perp/\tilde{\ell}]/\widetilde{C}$. 
We get the short exact sequence
\begin{equation}
\label{eq-equivariant-short-exact-sequence}
0\rightarrow \widetilde{C} \rightarrow \tilde{\ell}^\perp/\tilde{\ell}\rightarrow \widetilde{Q} \rightarrow 0
\end{equation}
of $\fS_n$-equivariant coherent sheaves  over $\PP(T^*S^n)$. $\widetilde{C}$ is isomorphic, as a $\fS_n$-equivariant 
sheaf, to $\tilde{\ell}\otimes_\ComplexNumbers W$, where $W$ is the reflection representation of $\fS_n$.  
If the torsion subsheaf of $\widetilde{Q}$ is non-zero, then its support has codimension $\geq 2$. But $\widetilde{C}$ is locally free
and hence reflexive. Consequently, 
the sheaf $\widetilde{Q}$ is torsion free, by Lemma \ref{lemma-reflexive-criterion} (\ref{lemma-item-iff}).
The dual sheaf $\widetilde{Q}^*$ is isomorphic to $\widetilde{C}$, since $\widetilde{C}$ 
is a Lagrangian subsheaf with respect to the $\fS_n$-invariant symplectic structure on $\tilde{\ell}^\perp/\tilde{\ell}$.
We conclude that neither $\widetilde{C}$ nor $\widetilde{Q}$ 
admit any non-trivial saturated $\fS_n$-equivariant subsheaf of lower rank, by Lemmas 
\ref{lemma-vanishing-of-global-sections-of-symmetric-powers} and \ref{lemma-on-equivariant-subsheaves}.

Assume that $F$ is  a non-trivial saturated
subsheaf of $\ell^\perp/\ell$  of rank $<2n-2$. 
Then $F$ restricts to a subsheaf of the restriction of  $\ell^\perp/\ell$ to $\PP(T^*U)$. 
Now $q^*\PP(T^*U)$ is isomorphic to $\PP(T^*\widetilde{U})$ 
yielding an \'{e}tale morphism $\tilde{q}:\PP(T^*\widetilde{U})\rightarrow \PP(T^*U)$, 
and 
$\tilde{q}^*F$ extends to a non-trivial saturated $\fS_n$-invariant  subsheaf $\widetilde{F}$ of 
$\tilde{\ell}^\perp/\tilde{\ell}$ of rank $\leq 2n-3$. 
We may assume that $\widetilde{F}$
has rank $\leq n-1$, possibly after replacing $\widetilde{F}$ with its symplectic-orthogonal subsheaf. 

If $\widetilde{F}$ is not contained in $\widetilde{C}$, then the natural homomorphism $h:\widetilde{F}\rightarrow \widetilde{Q}$ 
is $\fS_n$-equivariant and it does not vanish. Its image is an equivariant subsheaf of $\widetilde{Q}$ and it thus must have rank $n-1$.
The support of the quotient $\widetilde{Q}/h(\widetilde{F})$ has codimension $\geq 2$, since $\PP(T^*S^n)$ does not contain any 
effective divisor, by Lemma \ref{lemma-vanishing-of-global-sections-of-symmetric-powers}. 
The sheaf $\widetilde{F}$ is reflexive, being a saturated subsheaf of a locally free sheaf.
The quotient $\widetilde{Q}/h(\widetilde{F})$ must thus vanish, 
by Lemma \ref{lemma-reflexive-criterion} (\ref{lemma-item-torsion-free-F}),
since  $\widetilde{Q}$ is torsion free.
Hence, $h$ is an isomorphism and the short exact sequence (\ref{eq-equivariant-short-exact-sequence}) splits. 
But the bundle $\tilde{\ell}^\perp/\tilde{\ell}$ restricts to a slope-stable bundle with trivial determinant 
over every $\PP^{2n-1}$ fiber of $p_n$ \cite[Lemma 7.4]{torelli}. A contradiction. Hence, $\widetilde{F}$ is contained in $\widetilde{C}$.
Consequently, $\widetilde{F}=\widetilde{C}$, since the latter does not have any $\fS_n$-equivariant subsheaf
of lower rank. 
This completes the proof of Proposition \ref{prop-ell-perp-mod-ell-is-stable}.
\hide{
The rest of the proof is by induction on $n$.  The case $n=2$ is the initial step.

\underline{Case $n=2$:} In this case $\tilde{\ell}^\perp/\tilde{\ell}$ has rank $2$ and $\widetilde{F}$
is a reflexive sheaf of rank $1$, hence locally free.
Now $\Pic(\PP(T^*S^n))$ is cyclic, generated by the line-bundle $\tilde{\ell}$. Hence,
$\widetilde{F}$ is isomorphic to $\tilde{\ell}^N$, for some integer $N$.
The bundle $\tilde{\ell}^\perp/\tilde{\ell}$ restricts to a slope-stable bundle with trivial determinant 
over every $\PP^3$ fiber of $p_2$ \cite[Lemma 7.4]{torelli}. Hence, $N$ is positive.
We conclude that the $\fS_2$-invariant subspace
$H^0(\PP(T^*S^2),[\tilde{\ell}^\perp/\tilde{\ell}]\otimes \tilde{\ell}^{-N})^{\fS_2}$
does not vanish, for some positive integer $N$.
Consequently, $H^0\left(S^2,p_{2,*}\left\{\left[
\tilde{\ell}^\perp/\tilde{\ell}\right]\otimes \tilde{\ell}^{-N}
\right\}\right)^{\fS_2}$ does not vanish. We prove that this space  vanishes, if $N\neq 3$, and it is one dimensional if $N=3$, leading to the desired equality $F=C$.

Let us calculate 
the push-forward $p_{2,*}\left\{\left[\tilde{\ell}^\perp/\tilde{\ell}\right]\otimes \tilde{\ell}^{-N}\right\}$. 
We have the short exact sequence
\[
0\rightarrow \tilde{\ell}^\perp \otimes \tilde{\ell}^{-N} \rightarrow 
p_2^*T^*S^2\otimes \tilde{\ell}^{-N}\rightarrow \tilde{\ell}^{-N-1}\rightarrow 0.
\]
Identify $TS^2$ with $T^*S^2$ via $\sigma_2$. 
Then $p_{2,*}\left\{\tilde{\ell}^\perp\otimes \tilde{\ell}^{-N}\right\}$ is the kernel of
\[
T^*S^2\otimes \Sym^N(T^*S^2)\rightarrow \Sym^{N+1}(T^*S^2).
\]
The latter is the vector bundle associated to $T^*S^2$ via the Schur functor for the partition $(N,1)$ of $N+1$ 
and is consequently a direct summand of $\Wedge{2}T^*S^2\otimes \Sym^{N-1}(T^*S^2)$ \cite[Cor. 2 in Sec. 8.3, or Pieri Formula (4) in Sec. 2.2 ]{fulton}.
We conclude that $p_{2,*}\left\{\left[\tilde{\ell}^\perp/\tilde{\ell}\right]\otimes \tilde{\ell}^{-N}\right\}$ is a direct summand of
\[
\left[\Wedge{2}T^*S^2/\StructureSheaf{S^2}\cdot \sigma_2\right]\otimes \Sym^{N-1}(T^*S^2).
\]
It suffices to prove that the above vector bundle does not have $\fS_2$-invariant global sections.

The bundle $\Wedge{2}T^*S^2$ is isomorphic to the direct sum
\[
[\pi_1^*\Wedge{2}T^*S] \  \oplus \  
[\pi_2^*\Wedge{2}T^*S] \ \oplus \ 
[\pi_1^*T^*S\otimes \pi_2^*T^*S].
\]
Hence, $\Wedge{2}T^*S^2/\StructureSheaf{S^2}\cdot\sigma_2$ is isomorphic to the direct sum of
$\pi_1^*T^*S\otimes \pi_2^*T^*S$ and  the trivial line-bundle
$L:=\StructureSheaf{S^2}\cdot\sigma^-$, where the two form $\sigma^-$ spans
a non-trivial character of $\fS_2$. The vector bundle
$\Sym^{N-1}(T^*S^2)$ is isomorphic to the direct sum
\[
\oplus_{p=0}^{N-1}\left[\Sym^p(\pi_1^*T^*S)\otimes \Sym^{N-1-p}(\pi_2^*T^*S)\right].
\]

The vector spaces $H^0(S,\Sym^kT^*S)$ vanish, for all positive $k$, by Lemma
\ref{lemma-vanishing-of-global-sections-of-symmetric-powers}.
We conclude that $L\otimes \Sym^{N-1}(T^*S^2)$ does not have non-zero
$\fS_2$-invariant sections, for any positive $N$.

The vector bundle $T^*S\otimes \Sym^k(T^*S)$ does not have any non-zero global sections, for $k>1$, by 
Lemma 
\ref{lemma-vanishing-of-global-sections-of-symmetric-powers} and the isomorphism
$T^*S\otimes \Sym^k(T^*S)\cong \Sym^{k-1}(T^*S)\oplus \Sym^{k+1}(T^*S)$.
Thus, the vector bundles
\[
\left[\pi_1^*T^*S\otimes \pi_2^*T^*S\right]
\otimes
\left\{
\oplus_{p=0}^{N-1}\left[\Sym^p(\pi_1^*T^*S)\otimes \Sym^{N-1-p}(\pi_2^*T^*S)\right]
\right\}
\]
have non-zero global sections only for $N=3$. In this case the one-dimensional space 
of global sections arises from the direct summand
\begin{equation}
\label{eq-N=3-summand-which-has-global-section}
\left[\pi_1^*T^*S\otimes \pi_2^*T^*S\right] \otimes
\left\{
\Sym^1(\pi_1^*T^*S)\otimes \Sym^1(\pi_2^*T^*S)
\right\},
\end{equation}
which contains a direct summand isomorphic to
$\pi_1^*\Wedge{2}T^*S\otimes \pi_2^*\Wedge{2}T^*S$.
This one dimensional space of global sections must be $\fS_2$-invariant and corresponds to the 
line subbundle $C$ of $\ell^\perp/\ell$ constructed above. 
We conclude that $C$ is the unique non-trivial saturated proper subsheaf of $\ell^\perp/\ell$.

\noindent
\underline{Induction step.}
Assume $n\geq 3$.
Let $\phi:\PP(T^*S^{n-1})\times S\rightarrow \PP(T^*S^{n-1})$ be the projection on the the first factor and let 
$a:\PP(T^*S^{n-1})\times S\rightarrow S^n$ be morphism, such that  $\pi_k\circ a$ is the projection on the second factor 
and $p_{n-1}\circ \phi=\pi_{\hat{k}}\circ a$. Set $W:=\phi^*\tilde{\ell}_{n-1}\oplus a^*\pi_k^*T^*S$. 
We get the commutative diagram
\[
\xymatrix{
\PP(W) \ar[r]^b\ar[d]_\psi & \PP(T^*S^n) \ar[r]^{p_n} & S^n \ar[d]^{\pi_{\hat{k}}}
\\
\PP(T^*S^{n-1})\times S \ar[urr]^a \ar[r]_\phi & \PP(T^*S^{n-1}) \ar[r]_{p_{n-1}} & S^{n-1},
}
\]
where $\psi$ is the natural projection.
Note that $W$ is naturally a subbundle of $a^*T^*S^n$ and so the tautological line subbundle $\StructureSheaf{\PP(W)}(-1)$
of $\psi^*W$ is also a line-subbundle of $\psi^*a^*T^*S^n$, inducing the unique morphism $b$, such that $b^*\tilde{\ell}=\StructureSheaf{\PP(W)}(-1)$, as line subbundles of $a^*T^*S^n$.
The morphism $b$ is in fact the blow-up of $\PP(T^*S^n)$ along the subbundle $\PP(\pi_k^*T^*S)$ and $\psi$ is the projection from $\PP(\pi_k^*T^*S)$. 

Let $\Phi\subset \tilde{\ell}^\perp \subset p_n^*T^*S^n$ be the saturated subsheaf containing $\tilde{\ell}$, such that 
$\Phi/\tilde{\ell}=\widetilde{F}$. Then the rank of $\Phi$ is $\leq n$. 

\noindent
\underline{Case $n>3$ or [$n=3$ and $\rank(\widetilde{F})=1$] of the induction step:}
Set $\widetilde{Q}:=[\tilde{\ell}^\perp/\tilde{\ell}]/\widetilde{F}$ and let $\widetilde{Q}_{sing}$
be the closed analytic subset of $\PP(T^*S^n)$, where  $\widetilde{Q}$ is not locally free. 
The codimension of $\widetilde{Q}_{sing}$ is at least $2$, since $\widetilde{Q}$ is torsion free, by the assumption that $\widetilde{F}$ is saturated. 
Hence, the restriction of $\PP(T^*S^n)$ to $S^{n-1}\times \{x\}$, considered as a subvariety of 
$\PP(T^*S^n)$, is not contained in $\widetilde{Q}_{sing}$, for all but finitely many points $x$ of $S$. 
Choose such a point $x\in S$. Then the saturated restriction $G$ of $\widetilde{F}$ to 
$\PP(T^*S^n\restricted{)}{S^{n-1}\times \{x\}}$  has the same rank as $\widetilde{F}$.

Next, note that the restriction of $T^*S^n$ to $S^{n-1}\times \{x\}$ is the direct sum of $T^*S^{n-1}$ and the trivial vector bundle with fiber
$T^*_xS$. Let $A$ be a line in $T^*_xS$ and denote by $\underline{A}$ the trivial line bundle 
over $S^{n-1}\times \{x\}$ with fiber $A$.
Set $D_A:=\PP[T^*(S^{n-1}\times \{x\})\oplus\underline{A}],$
regarded as a divisor in $\PP(T^*S^n\restricted{)}{S^{n-1}\times \{x\}}$. The saturated restriction $G'_A$ of $G$ to 
$D_A$ has the same rank as that of $G$, hence also as the rank of $\widetilde{F}$.
The saturated restriction $G''_A$ of $G'_A$ to the divisor $\PP[T^*(S^{n-1}\times \{x\})]$ in $D_A$ has again the same rank as $\widetilde{F}$. 
The restriction of $\tilde{\ell}^\perp/\tilde{\ell}$ to $\PP[T^*(S^{n-1}\times\{x\})]$ is the $\fS_{n-1}$-equivariant 
direct sum $[\tilde{\ell}_{n-1}^\perp/\tilde{\ell}_{n-1}] \oplus B$, where 
$\tilde{\ell}_{n-1}^\perp/\tilde{\ell}_{n-1}$
is the  sheaf  analogous to $[\tilde{\ell}^\perp/\tilde{\ell}]$ associated with $n-1$ and 
$B$ is the trivial vector bundle over $\PP[T^*(S^{n-1}\times\{x\})]$ with fiber $T^*_{x}S$.

\begin{claim}
$G''_A$ is not contained in the trivial direct summand $B$ of the restriction
of $[\tilde{\ell}^\perp/\tilde{\ell}]$.
\end{claim}

\begin{proof}
%
We prove first that $\det(G')$  is isomorphic to $\tilde{\ell}_{n-1}^d$, for some positive integer $d$.
Let $\Sigma$ be the singular locus of $\widetilde{F}$. Then $\Sigma$ has codimension $\geq 3$ in
$\PP(T^*S^n)$. Thus $\Sigma$ does not contain 
$\pi_{\widehat{n}}^*\PP(T^*S^{n-1})$. The latter is isomorphic to
$S\times \PP(T^*S^{n-1})$ and 
it does not contain any effective divisor,
by Lemma \ref{lemma-vanishing-of-global-sections-of-symmetric-powers}.
We conclude that $\Sigma\cap \pi_{\widehat{n}}^*\PP(T^*S^{n-1})$ has codimension $\geq 2$ in
$\pi_{\widehat{n}}^*\PP(T^*S^{n-1})$. It follows that $\det(G)$ is equal to the restriction
of $\det(\widetilde{F})$.
A repeat of this argument shows that $\det(G')$ is equal to the restriction of
$\det(\widetilde{F})$.

The restriction of $\tilde{\ell}^\perp/\tilde{\ell}$ to a $\PP^{2n-1}$ fiber of $p_n$ is a slope stable vector bundle
with a trivial determinant line-bundle \cite[Lemma 7.4]{torelli}.
Hence, the restriction of $\det(\widetilde{F})$ to each fiber is of the form $\StructureSheaf{\PP^{2n-1}}(d)$,
for some negative integer $d$. We conclude that $\det(G')$ is isomorphic to $\tilde{\ell}_{n-1}^d$, 
for some positive $d$.
In particular, $G'$ is not equal to $B$.

It remains to be proved that $G'$ is not a rank $1$ saturated subsheaf of the trivial rank $2$ vector bundle
$B$ over $\PP(T^*S^{n-1})$. 
If this was the case, then $G'$ would be isomorphic to $\tilde{\ell}_{n-1}^d$ and its embedding in $B$ would correspond to a non-zero
section of $H^0(\PP(T^*S^{n-1}),B\otimes \tilde{\ell}_{n-1}^{-d})$. But the latter is isomorphic to 
$T^*_{x_n}S\otimes_{\ComplexNumbers}H^0(\PP(T^*S^{n-1}),\tilde{\ell}_{n-1}^{-d})$, which  vanishes 
by Lemma \ref{lemma-vanishing-of-global-sections-of-symmetric-powers}.
\end{proof}

Let $\widetilde{F}_{n-1}$ be the saturation of the 
projection of $G'$ to the direct summand $\tilde{\ell}_{n-1}^\perp/\tilde{\ell}_{n-1}$.
The above claim shows that the projection is non-trivial. 
If $n>3$, then $\rank(G')=\rank(\widetilde{F})\leq n-1<2n-4=\rank[\tilde{\ell}_{n-1}^\perp/\tilde{\ell}_{n-1}]$.
Thus, the rank of $\widetilde{F}_{n-1}$ is lower than the rank of $\tilde{\ell}_{n-1}^\perp/\tilde{\ell}_{n-1}$.
If $n=3$ and $\rank(\widetilde{F})=1$, then $\widetilde{F}_{n-1}$ has rank one. 

\underline{Case $n=3$ and $\rank(\widetilde{F})=2$ of the induction step:} 
The above construction yields a $\fS_2$-invariant saturated subsheaf  $G'$ of
$[\tilde{\ell}_{2}^\perp/\tilde{\ell}_{2}]\oplus B$, which is not equal to 
the trivial summand $B:=T_{x_3}S\otimes_\ComplexNumbers \StructureSheaf{\PP(T^*S^2)}$, and $\det(G')$ is isomorphic to $\tilde{\ell}^d_2$, for some positive integer $d$.
The $\fS_2$-equivariant  projection 
\[
h \ : \ G' \ \ \rightarrow \ \ [\tilde{\ell}_{2}^\perp/\tilde{\ell}_{2}]
\]
is thus of positive rank. 
We claim that the rank of $h$ is strictly less than $2$. Indeed, otherwise the cokernel of $h$
will be supported on an effective divisor in $\PP(T^*S^2)$. But such a divisor does not exist. Thus, 
$h$ has rank $1$. Define $\widetilde{F}_2$ to be the saturation of its image. 

In all cases, given  a non-trivial subsheaf $F_n$ of $\ell^\perp/\ell$ over $\PP(T^*S^{[n]})$ of rank less than $2n-2$,
we constructed a non-trivial $\fS_{n-1}$-invariant 
subsheaf $\widetilde{F}_{n-1}$ of $\tilde{\ell}_{n-1}^\perp/\tilde{\ell}_{n-1}$ of rank less than $2n-4$ 
over $\PP(T^*S^{n-1})$.
Let $F_{n-1}$ 
be the corresponding saturated subsheaf of $\ell^\perp/\ell$ over $\PP(T^*S^{[n-1]})$.
The induction hypothesis implies that such a sheaf $F_{n-1}$ does not exist. 
Hence,  such a subsheaf $F_n$ can not exist.
This completes the proof of Proposition \ref{prop-ell-perp-mod-ell-is-stable}.
}
\end{proof}


\end{document}